\documentclass[11pt]{article}
\usepackage{amsmath,amsthm,amsfonts,amssymb,graphicx,graphpap}
\usepackage[hypertex]{hyperref}
\usepackage{xcolor}
\usepackage[all]{xy}
\usepackage{thmtools, thm-restate}
\usepackage[a4paper,textwidth=15cm,textheight=25cm]{geometry}

\title{On the intersection of tame subgroups in groups acting on trees}
\author{Konstantinos Lentzos and Mihalis Sykiotis} \makeatletter

\newcommand{\dsp}{\displaystyle}
\newtheorem{thm}{Theorem}[section]

\newtheorem{lem}[thm]{Lemma}
\newtheorem{cor}[thm]{Corollary}

\theoremstyle{definition}
\newtheorem{defi}[thm]{Definition}

\theoremstyle{remark}
\newtheorem{rem}[thm]{Remark}

\def\classification{\@ifnextchar [{\@xfootnotetext}%
{\begingroup\let\protect\noexpand\xdef\@thefnmark{}
\endgroup\@footnotetext}}

\begin{document}


\maketitle

\begin{abstract} Let $G$ be a group acting on a tree $T$ with finite
edge stabilizers of bounded order. We provide, in some very
interesting cases, upper bounds for the complexity of the
intersection $H\cap K$ of two tame subgroups $H$ and $K$ of $G$ in
terms of the complexities of $H$ and $K$. In particular, we obtain
bounds for the Kurosh rank $Kr(H\cap K)$ of the intersection in
terms of Kurosh ranks $Kr(H)$ and $Kr(K)$, in the case where $H$ and
$K$ act freely on the edges of $T$.

\end{abstract}

\section{Introduction}
In 1954, Howson \cite{Ho} showed that the intersection of two
finitely generated subgroups $H$ and $K$ of a free group $F$ is also
finitely generated and provided an upper bound for the rank $r(H\cap
K)$ of $H\cap K$ in terms of $r(H)$ and $r(K)$. The Hanna Neumann
conjecture, proved independently by Friedman \cite{Fr} and Mineyev
\cite{Mi} in 2011, says that $\overline{r}(H\cap K)\leq
\overline{r}(H)\overline{r}(K)$, where
$\overline{r}(A)=\max\{0,r(A)-1\}$ is the reduced rank of a free
group $A$.

For free products the situation is analogous. Let $\Gamma$ be a
group. The \textsl{Kurosh rank}, denoted $Kr(\Gamma)$, of a free
product decomposition $\Gamma=\ast_{i\in I}G_{i}$ of $\Gamma$ is
defined to be the number of (non-trivial) factors $G_{i}$. By the
Kurosh subgroup theorem, any subgroup $H$ of $\Gamma$ inherits a
free product decomposition $H=\ast_{j\in J}H_{j}\ast F$, where each
$H_{j}$ is non-trivial and conjugate to a subgroup of a free factor
of $\Gamma$ and $F$ is a free group. The (subgroup) Kurosh rank of
$H$ of $\Gamma$ with respect to the above splitting of $\Gamma$, is
the sum $|J|+r(F)$, which we again denote by $Kr(H)$. The reduced
Kurosh rank of $H$ is defined to be
$\overline{K}r(H)=\max\{0,Kr(H)-1\}$.

Free products also have the Howson property, in the following sense:
if $H$, $K$ are subgroups of $\Gamma$ of finite Kurosh rank, then
$H\cap K$ also has finite rank (see \cite[Theorem 2.13 (1)]{Sy} for
a proof). In \cite{Iv}, Ivanov proved that if $\Gamma$ is torsion
free, then $\overline{K}r(H\cap K)\leq 2
\overline{K}r(H)\overline{K}r(K)$. It is shown in \cite{AnMS}, that
if $\Gamma$ is right-orderable, then the coefficient $2$ can be
replaced by $1$.

The problem of finding bounds for the ``rank" of the intersection of
subgroups in free products and more generally in groups satisfying
the Howson property has also been considered in
\cite{So,BCK,Iv0,DI,DI1,Za1,Za2,ArSS}.

In this paper, we obtain, under appropriate hypotheses, bounds for
the complexity of the intersection of tame subgroups in groups
acting on trees with finite edge stabilizers.

Let $G$ be a group acting on a (simplicial) tree $T$ without
inversions. A vertex $v$ of $T$ is called \textsl{($G$-) degenerate}
if $G_{v} = G_{e}$ for some edge $e$ incident to $v$. The
corresponding vertex $[v]_{G}$ of the quotient graph $T/G$ is also
called degenerate. Let $H$ be a subgroup of $G$. We denote by
$r(T/H)$ the rank of the fundamental group of $T/H$ and by
$V_{ndeg}(T/H)$ the set of $H$-non-degenerate vertices of $T/H$. The
\textsl{complexity} $C_{T}(H)$ of $H$ with respect to $T$ is defined
to be the sum $C_{T}(H)= r(T/H) + |V_{ndeg}(T/H)|\in [0,\infty]$, if
$H$ contains hyperbolic elements, and $1$ otherwise. The
\textsl{reduced complexity} of $H$ with respect to $T$, is defined
as $\overline{C}_{T}(H)=max\{{C}_{T}(H)-1,0\}$. The subgroup $H$ of
$G$ is called \textsl{tame} if either $H$ fixes a vertex, or $H$
contains a hyperbolic element and the quotient graph $T_{H}/H$ is
finite, where $T_{H}$ is the unique minimal $H$-invariant subtree of
$T$. By \cite[Theorem 2.13]{Sy}, if each edge stabilizer is finite,
then the intersection of two tame subgroups $H$, $K$ of $G$ is again
tame. In the case where $H\cap K$ fixes a vertex, we obviously have
$\overline{C}_{T}(H\cap K)\leq \overline{C}_{T}(H)\cdot
\overline{C}_{T}(K)$.

Finitely generated subgroups are examples of tame subgroups. In the
case of free products, finite Kurosh rank implies tameness (see
Lemma \ref{lem0}) and the complexity of a non-trivial subgroup is
exactly its Kurosh rank (see section \ref{prelim} for more details).
Our first main result is the following.

\begin{restatable*}{thm}{firstThmOne}\label{ThmOne} Let $G$ be a group acting on a tree $T$
with finite quotient and finite stabilizers of edges and let $H$,
$K$ be tame subgroups of $G$ such that $H\cap K$ does not fix a
vertex of $T$.
\begin{enumerate}
\item If $T_{H}/H$ and $T_{K}/K$ do not contain degenerate vertices
of valence two, then
$$\overline{C}_{T}(H\cap K)\leq \big(6NM+12(M-1)N\big)\cdot\overline{C}_{T}(H)\cdot
\overline{C}_{T}(K),$$ where $N=\max\big\{|G_{x}\cap HK |\,:\,x\in
ET\big\}$ and $M=\max\{M_{H},M_{K}\}\leq \max\big\{|G_{x}|\,:\,x\in
ET\big\}$.

\item Suppose $H$ and $K$ satisfy the following property: for each $H$-degenerate
(resp. $K$-degenerate) vertex $v$ of $T$, the stabilizer $H_{v}$
(resp. $K_{v}$) stabilizes each edge in the star of $v$. Then
$$\overline{C}_{T}(H\cap K)\leq 6N\cdot\overline{C}_{T}(H)\cdot
\overline{C}_{T}(K).$$  In particular, if $H,K$ act freely on the
edges of $T$, then
$$\overline{K}_{T}(H\cap K)\leq
6N\cdot\overline{K}_{T}(H)\cdot \overline{K}_{T}(K).$$
\end{enumerate}
\end{restatable*}

In the special case where both $H$ and $K$ act freely on $T$, the
above inequality was proved by Zakharov in \cite{Za1}.

Now let $G=\ast_{A}G_{i}\ast F$ be the free product of the
amalgamated free product of $G_{i}$'s with a finite amalgamated
subgroup $A$ and $F$, such that $A$ is normal in each $G_{i}$.
Following Dicks and Ivanov \cite{DI}, we define
$a_{3}(G_{i}/A)=\min\big\{|\Gamma|\,:\,\Gamma \textrm{ is a subgroup
of } G_{i}/A \textrm{ with } |\Gamma|\geq 3 \big\}$ and
$\theta(G_{i}/A)=\Big\{\frac{a_{3}(G_{i}/A)}{a_{3}(G_{i}/A)-2}\Big\}\in
[1,3]$, where $\frac{\infty}{\infty-2}:=1$.

We represent $G$ as the fundamental group of a graph of groups
$(\mathcal{G},\Psi)$, where $\Psi$ is the wedge of copies of $[0,1]$
(one copy for each factor $G_{i}$) and a bouquet of circles (one for
each free generator of $F$). To each copy of $[0,1]$ and to the
wedge point we associate the group $A$, and to each circle we
associate the trivial group. To each of the remaining vertices we
associate a factor $G_{i}$. Let $T$ be the corresponding universal
tree.

\begin{restatable*}{thm}{firstThmTwo}\label{thm2} Let $G=\ast_{A}G_{i}\ast F$
be the free product of the amalgamated free product of $G_{i}$'s
with a finite amalgamated subgroup $A$ and $F$, such that $A$ is
normal in each $G_{i}$. We consider the natural action of $G$ on $T$
defined above. Suppose that $H$ and $K$ are tame subgroups (with
respect to $T$) of $G$ which act freely on the edges of $T$. Then
$H\cap K$ is tame and
$$\overline{K}_{T}(H\cap K)\leq 2\cdot \theta\cdot
N\cdot\overline{K}_{T}(H)\cdot \overline{K}_{T}(K)\leq 2\cdot \theta
\cdot |A|\cdot\overline{K}_{T}(H)\cdot \overline{K}_{T}(K)\,,$$
where $\theta=\max\{\theta(G_{i}/A)\,:\,i\in I\}$ and
$N=\max\big\{|gAg^{-1}\cap HK |\,:\,g\in G\big\}$.
\end{restatable*}

As an immediate corollary we obtain the main result of \cite{Iv}
mentioned above.

It should be noted that the arguments in the proof of Theorem
\ref{thm2}, work in a slightly more general setting as well. Thus,
with essentially the same proof, we obtain Theorem \ref{thm3} (see
also Remark \ref{rem2}): If $H$, $K$ are tame subgroups of a free
product $\ast_{A}G_{i}$ with a finite and normal amalgamated
subgroup $A$, then $\overline{C}_{T}(H\cap K)\leq 2\cdot \theta
\cdot |A\cap HK|\cdot\overline{C}_{T}(H)\cdot \overline{C}_{T}(K),$
where $\theta=\max\{\theta(G_{i}/A)\,:\,i\in I\}$ and $T$ is defined
as above for $F=1$.

After posting the first version of this paper on the arXiv, the
authors learned from A. Zakharov that he, in collaboration with S.
Ivanov, had also recently obtained (unpublished) upper bounds for
the Kurosh rank of the intersection of free product subgroups in
groups acting on trees with finite edge stabilizers.

\textbf{Acknowledgements.} We are grateful to Dimitrios Varsos for
many useful discussions and comments. We are also grateful to the
anonymous referee for careful reading of the manuscript and pointing
out a mistake in an earlier version.
\section{Preliminaries}\label{prelim}

To fix our notation, we first recall the definition of a graph.
\begin{defi} A \textsl{graph} $X$ consists of a (nonempty) set of vertices $VX$, a set of edges $EX$, a
fixed-point free involution $^{-1}:EX \rightarrow EX$ ($e\mapsto
e^{-1}$) and a map $i:EX\rightarrow VX$. The vertex $i(e)$ is called
the \textsl{initial} vertex of the edge $e$. The \textsl{terminal}
vertex $t(e)$ of $e$ is defined by $t(e)=i(e^{-1})$.
\end{defi}

Throughout, let $G$ be a group acting on a (simplicial) tree $T$
(without inversions, i.e. $ge\neq e^{-1}$ for any $g\in G$ and $e\in
EX$). By Bass-Serre theory, for which we refer to
\cite{DicksDun,Se}, this is equivalent to saying that $G$ is the
fundamental group of the corresponding graph of groups
$(\mathcal{G},T/G)$. If $x\in T$, we denote by $[x]_{G}$ the
$G$-orbit of $x$ and by $G_{x}$ its stabilizer. An element $g\in G$
is \textsl{elliptic} if it fixes a vertex of $T$ and
\textsl{hyperbolic} otherwise. If $H$ is a subgroup of $G$
containing a hyperbolic element, then there is a unique minimal
$H$-invariant subtree $T_{H}$ which is the union of the axes of the
hyperbolic elements of $H$.

We recall that a subgroup $H$ of $G$ is called \textsl{tame} if
either $H$ fixes a vertex, or $H$ contains a hyperbolic element and
the quotient graph $T_{H}/H$ is finite. By \cite[Prop. 2.2]{Sy}, the
subtree $T_{H}$ is a ``core" for the action of $H$ on $T$ in the
sense that $r(T/H)+|V_{ndeg}(T/H)| =
r(T_{H}/H)+|V_{ndeg}(T_{H}/H)|$, i.e. $C_{T}(H)=C_{T_{H}}(H)$. From
this it follows that the complexity of a tame subgroup is finite.

Finitely generated subgroups of $G$ are examples of tame subgroups,
since a finitely generated group $\Gamma$ acting by isometries on
$T$, either fixes a point of $T$ or else contains a hyperbolic
isometry and the quotient graph $T_{\Gamma}/\Gamma$ is finite.

\begin{rem}\label{rem1} We note that if the $G$-stabilizer of each edge is finite
and there is a bound on their orders, then any subgroup of $G$
consisting of elliptic elements fixes a vertex of $T$ (\cite[Lem.
2.5]{Sy}).
\end{rem}
If we restrict attention to subgroups $H$ of $G$ that act
edge-freely on $T$, then the \textsl{Kurosh rank} $K_{T}(H)$ of $H$
(with respect $T$) is defined to be the complexity $C_{T}(H)$ of
$H$.

Let $\Gamma=\ast_{i\in I}G_{i}$ be a free product and $H$ a subgroup
of $\Gamma$. By the Kurosh subgroup theorem, $H=\ast_{i\in
I,g_{i}}(H\cap g_{i}G_{i}g_{i}^{-1})\ast F$, where for each $i$,
$g_{i}$ ranges over a set of double coset representatives in
$G_{i}\backslash \Gamma/H$ and $F$ is a free group intersecting each
conjugate $gG_{i}g^{-1}$ trivially. The \textsl{(subgroup) Kurosh
rank} of $H$ with respect the above free product decomposition of
$\Gamma$, denoted by $Kr(H)$, is the sum $|\Lambda| +rank(F)$, where
$|\Lambda|$ is the number of all non-trivial factors $H\cap
g_{i}G_{i}g_{i}^{-1}$. Note that the Kurosh rank of $\Gamma$ is the
number of non-trivial factors $G_{i}$.

It is not difficult to verify that the numbers $|\Lambda|$,
$rank(F)$ depend only on $H$ and the given free product
decomposition of $\Gamma$. In fact, if $T$ is any $\Gamma$-tree
corresponding to the given decomposition of $\Gamma$, then the
Kurosh rank of $H$ with respect to $\Gamma=\ast_{i\in I}G_{i}$ is
equal to the Kurosh rank $K_{T}(H)$ of the associated free product
decomposition of $H$ coming from the action of $H$ on $T$. Thus, if
$H$ is non-trivial, then $Kr(H)=K_{T}(H)=C_{T}(H)$.
\begin{lem}\label{lem0} Let $G$ be a group acting on a tree $T$ and $H$ a
subgroup of $G$ that act edge-freely on $T$. If $K_{T}(H)<\infty$,
then $H$ is tame.
\end{lem}

\begin{proof} It suffices to consider the case when $H$ contains a
hyperbolic element. Let $\pi:T\rightarrow T/H$ be the natural
projection given by $\pi(x)=[x]_{H}$. Since $K_{T}(H)<\infty$, there
are finitely many vertices $v_{1},\ldots,v_{n}$ of $T/H$ with
non-trivial group and finitely many edges $e_{1},\ldots,e_{m}$ of
$T/H$ such that $X=T/H\setminus \{e_{1},\ldots,e_{m}\}$ is a maximal
tree of $T/H$. Let $Y$ be the finite subgraph of $T/H$ consisting of
$\{e_{1},\ldots,e_{m}\}$ and all geodesics in $X$ between endpoints
of the $e_{i}'s$ and $v_{1},\ldots,v_{n}$. We claim that
$\pi^{-1}(Y)$ is connected. To see this, let $p=x_{1}\cdots x_{k}$
be a reduced path connecting vertices of $\pi^{-1}(Y)$ such that no
edge of $p$ lies in $\pi^{-1}(Y)$. Then $\pi(p)$ is contained in the
complement $T/H\setminus Y$ of $Y$. Since $Y$ contains the edges
$e_{1},\ldots,e_{m}$, each component $C$ of $T/H\setminus Y$ is a
tree, and it is not difficult to see that $C$ intersects
$T/H\setminus Y$ in only one vertex. It follows that there is an
index $i$ such that $\pi(x_{i})=\pi(x_{i+1})^{-1}$. This means that
$hx_{i}=x_{i+1}^{-1}$ for some $h\in H$ and hence $h$ fixes the
initial vertex $v$ of $x_{i}$. From the construction of $Y$, $v$ is
degenerate and therefore $h=1$, which contradicts the choice of $p$.

Thus, $\pi^{-1}(Y)$ is a connected $H$-invariant subgraph of $T$. It
follows that $T_{H}\subseteq \pi^{-1}(Y)$. We conclude that
$T_{H}/H$ is finite, being a subgraph of $Y$.
\end{proof}

\section{Proofs of the main results}

Let $Y$ be a graph and $v$ a vertex of $Y$. The \textsl{star} of
$v$, denoted $Star_{Y}(v)$, is the set of edges of $Y$ with initial
vertex $v$, i.e. $Star_{Y}(v)=\{e\in EY\,|\,i(e)=v\}$. The
\textsl{valence} or \textsl{degree} of $v$ in $Y$, denoted
$\deg_{Y}(v)$, is the number of edges in the star of $v$.

\begin{lem}\label{lem1} Let $G$ be a group acting on a tree $T$,
let $H$ be a tame subgroup of $G$ containing hyperbolic elements and
let $\widetilde{X}$ be the graph obtained from $X=T_{H}/H$ by
attaching a loop at each $H$-non-degenerate vertex. Then
\[\overline{C}_{T}(H)=\overline{r}(\widetilde{X})=\frac{1}{2}\sum\big(\deg_{\widetilde{X}}([v]_{H})-2\big)\,,\]
where the sum is taken over all vertices $[v]_{H}$ of
$\widetilde{X}$.

\end{lem}

\begin{proof} The reduced rank
of a graph is equal to the number of its (geometric-oriented) edges
minus the number of its vertices. The minimality of $T_{H}$ implies
that each vertex of $X$ of valence one is $H$-non-degenerate.
Therefore, every vertex of $\widetilde{X}$ has valence at least two.
Now an easy calculation shows that the sum
$\sum\big(\deg_{\widetilde{X}}([v]_{H})-2\big)$, over all vertices
$[v]_{H}$ of $\widetilde{X}$, is equal to
2$\overline{r}(\widetilde{X})$. By construction of $\widetilde{X}$,
we have $\overline{r}(\widetilde{X})=\overline{C}_{T}(H)$ which
completes the proof.
\end{proof}

\begin{lem}\label{lem2} Let $G$ be a group acting on a tree $T$ and let $A$ and $B$ be
subgroups of $G$ such that $A\subseteq B$. Suppose that $A$ and $B$
contain hyperbolic elements and that $v$ is a vertex of $T_{B}$. We
consider the graph map $\pi_{B}:T_{A}/A \longrightarrow T_{B}/B$
given by $\pi([x]_{A})=[x]_{B}$.
\begin{enumerate}
\item $|Star([v]_{A})|\leq |G_{v}\cap
B|\cdot |Star([v]_{B})|$ (provided that they are finite).
\item If, moreover, $B_{v}$ is $B$-degenerate and stabilizes each edge in $Star_{T_{B}}(v)$, then the
restriction $\pi_{B}: Star([v]_{A}) \longrightarrow Star([v]_{B})$
is an embedding.
\end{enumerate}
\end{lem}

\begin{proof}
Suppose that $[e_{1}]_{A}$ and $[e_{2}]_{A}$ are two edges in the
star of $[v]_{A}$ with $\pi([e_{1}]_{A})=\pi([e_{2}]_{A})$. Then
there are $a_{1},a_{2}\in A$ and $b\in B$ such that
$i(e_{1})=a_{1}v$, $i(e_{2})=a_{2}v$ and $e_{1}=be_{2}$. It follows
that $i(e_{1})=bi(e_{2})$ and thus $a_{1}v=ba_{2}v$. Hence
$a_{1}^{-1}ba_{2}\in G_{v}\cap B=B_{v}$. Now, if $[x]_{A}$ is an
edge in the star of $[v]_{A}$ with
$\pi([x]_{A})=\pi([e_{1}]_{A})=\pi([e_{2}]_{A})$, then as before
$i(x)=a_{x}v$ and $x=b_{x}e_{2}$ for some $a_{x}\in A$ and $b_{x}\in
B$. If we assume further that
$a_{1}^{-1}ba_{2}=a_{x}^{-1}b_{x}a_{2}$, then
$a_{1}^{-1}b=a_{x}^{-1}b_{x}$ and so
$[e_{1}]_{A}=[be_{2}]_{A}=[a_{1}a_{x}^{-1}b_{x}e_{2}]_{A}=[b_{x}e_{2}]_{A}=[x]_{A}$.
This means that each fiber of the restriction (on stars) has at most
$|G_{v}\cap B|$ elements, and the first assertion follows.

Now, if $B_{v}$ stabilizes each edge in $Star_{T_{B}}(v)$, then
$a_{1}^{-1}ba_{2}$ stabilizes $a_{2}^{-1}e_{2}$ and therefore
$[e_{1}]_{A}=[be_{2}]_{A}=[a_{1}a_{2}^{-1}e_{2}]_{A}=[e_{2}]_{A}$.
\end{proof}
In view of this lemma, we define $M_{B}:=\max\{|G_{v}\cap B|\,:\,v
\textrm{ is a B-degenerate vertex of T }\}$.

The following is our first main result.

\firstThmOne
\begin{proof} Since $H\cap K$ does not fix a vertex, it
follows from Remark \ref{rem1} that $H\cap K$, $H$ and $K$ contain
hyperbolic elements. Let $T_{H\cap K}$, $T_{H}$, $T_{K}$ be the
minimal subtrees of $T$ invariant under $H\cap K$, $H$, $K$,
respectively. Let $\pi_{H}:T_{H\cap K}/H\cap K \longrightarrow
T_{H}/H$ and $\pi_{K}:T_{H\cap K}/H\cap K \longrightarrow T_{K}/K$
be the natural projections (defined as in Lemma \ref{lem2}). We
consider the map $\pi=(\pi_{H},\pi_{K}):T_{H\cap K}/H\cap K
\longrightarrow T_{H}/H \times T_{K}/K$ given by $\pi([x]_{H\cap
K})=([x]_{H},[x]_{K})$. By \cite[Proposition 8.7]{Ba}, each fiber
$\pi^{-1}([x]_{H},[x]_{K})$, where $x$ is an edge or a vertex, has
exactly $|H_{x}\backslash G_{x}\cap HK / K_{x}|$ elements. It
follows that for each edge $x$ the fiber $\pi^{-1}([x]_{H},[x]_{K})$
has at most $N$ elements.

For convenience we simplify notation by setting $X=T_{H\cap K}/H\cap
K$, $Y=T_{H}/H$ and $Z=T_{K}/K$. As in Lemma \ref{lem1}, we
construct graphs $\widetilde{X}$, $\widetilde{Y}$ and
$\widetilde{Z}$, by attaching a loop at each non-degenerate vertex
of $X$, $Y$ and $Z$, respectively.

1) By Lemma \ref{lem1}, it suffices to show that
\begin{align}\label{ineq1}
\sum_{V\widetilde{X}}\big(\deg_{\widetilde{X}}([v]_{H\cap
K})-2\big)& \leq \big(3NM+6N(M-1)\big)
\sum_{V\widetilde{Y}}\big(\deg_{\widetilde{Y}}([v]_{H})-2\big) \cdot
\sum_{V\widetilde{Z}}\big(\deg_{\widetilde{Z}}([v]_{K})-2\big).
\end{align}
 For any pair of vertices $(a,b)\in Y\times Z$, we will show that
\begin{align}\label{ineq2} \sum_{v\in
\pi^{-1}(a,b)}\big(\deg_{\widetilde{X}}(v)-2\big)&\leq
\big(3NM+6N(M-1)\big)\cdot \big(\deg_{\widetilde{Y}}(a)-2\big) \cdot
\big(\deg_{\widetilde{Z}}(b)-2\big)
\end{align}
from which (\ref{ineq1}) follows. The rest of the proof follows
similar arguments to those given in \cite{BCK}, \cite{DI} and
\cite{Iv}. Let $\{v_{1},\ldots, v_{n}\}$ be the vertices of
$\pi^{-1}(a,b)$. Since the fiber of any edge of $Y\times Z$ contains
at most $N$ edges, we have
\begin{equation}\label{ineq3}
 \sum_{i=1}^{n} \deg_{X}(v_{i}) \leq N \cdot \deg_{Y}(a) \cdot
 \deg_{Z}(b).
\end{equation}

We consider three cases depending on whether or not $a$ and $b$ are
degenerate. \\
\textbf{Case 1.} Suppose that $a$ is $H$-non-degenerate and $b$ is
$K$-non-degenerate. Then $\deg_{Y}(a) =\deg_{\widetilde{Y}}(a) -2$,
$\deg_{Z}(b) = \deg_{\widetilde{Z}}(b) -2$ while $\deg_{X}(v_{i})$
is equal to $\deg_{\widetilde{X}}(v_{i}) -2$ or
$\deg_{\widetilde{X}}(v_{i})$. Hence
$$\sum_{i=1}^{n} \big(\deg_{\widetilde{X}}(v_{i})-2\big) \leq  \dsp \sum_{i=1}^{n} \deg_{X}(v_{i})
\leq  N \cdot \deg_{Y}(a)\cdot\deg_{Z}(b)= N\cdot
\big(\deg_{\widetilde{Y}}(a) -2\big)\cdot
\big(\deg_{\widetilde{Z}}(b) -2\big)\,.
$$
\textbf{Case 2.} Exactly one of $a$, $b$, say $b$, is degenerate.
Then each $v_{i}$ is $(H\cap K)$-degenerate as well, and thus
$\deg_{Y}(a) =\deg_{\widetilde{Y}}(a) -2$, $\deg_{X}(v_{i}) =
\deg_{\widetilde{X}}(v_{i})$ and
$\deg_{\widetilde{Z}}(b)=\deg_{Z}(b)
>2$. Also, by Lemma \ref{lem2}, for each $i$
we have $\deg_{X}(v_{i})\leq M\deg_{Z}(b)$.

If $n\leq N\cdot \deg_{Y}(a)$, then
\begin{align}
\sum_{i=1}^{n}
\big(\deg_{\widetilde{X}}(v_{i})-2\big)&=\sum_{i=1}^{n}\big(\deg_{X}(v_{i})-2\big)
\leq n\cdot \big(M\deg_{Z}(b)-2\big)\leq N\cdot \deg_{Y}(a) \big(M\deg_{Z}(b)-2\big)\nonumber \\
 &\leq N\cdot \deg_{Y}(a)\Big(M\big(\deg_{Z}(b)-2\big)+2(M-1)\Big)\nonumber \\
 & = NM
\big(\deg_{\widetilde{Y}}(a) -2\big) \big(\deg_{\widetilde{Z}}(b)
-2\big)+2N(M-1)\big(\deg_{\widetilde{Y}}(a) -2\big)\label{2a}\\
&\leq \big(NM+2N(M-1)\big)\cdot \big(\deg_{\widetilde{Y}}(a)
-2\big)\cdot \big(\deg_{\widetilde{Z}}(b) -2\big),\nonumber
 \end{align}
where the last inequality follows because
$\deg_{\widetilde{Z}}(b)>2$.

On the other hand, if $n\geq N\cdot \deg_{Y}(a)$, then
\begin{align}
\sum_{i=1}^{n}
\big(\deg_{\widetilde{X}}(v_{i})-2\big)&=\sum_{i=1}^{n}\deg_{X}(v_{i})-2n
\leq N\cdot \deg_{Y}(a)\cdot \deg_{Z}(b)-2N \cdot \deg_{Y}(a) \nonumber\\
 &=
 N \deg_{Y}(a)\cdot\big(\deg_{Z}(b)-2\big)
 = N
\big(\deg_{\widetilde{Y}}(a) -2\big)\big(\deg_{\widetilde{Z}}(b)
-2\big)\label{2b}.
\end{align}
\textbf{Case 3.} Finally, suppose that $a$, $b$ are degenerate in
$Y$, $Z$, respectively. Then each vertex $v_{i}$ is $(H\cap
K)$-degenerate as well and $\deg_{\widetilde{Y}}(a)=\deg_{Y}(a)
>2$, $\deg_{\widetilde{X}}(v_{i})=\deg_{X}(v_{i})$,
$\deg_{\widetilde{Z}}(b)=\deg_{Z}(b)>2$. Moreover, by Lemma
\ref{lem2}, $\deg_{X}(v_{i})\leq \min\{M\deg_{Y}(a),M\deg_{Z}(b)
\}$. Suppose that $\deg_{Z}(b)=\min\{\deg_{Y}(a),\deg_{Z}(b) \}$ and
hence $\deg_{Y}(a)=\max\{\deg_{Y}(a),\deg_{Z}(b) \}$ (the other case
is handled in the same way).

If $n\leq N\cdot \deg_{Y}(a)$, then
$$\sum_{i=1}^{n}
\big(\deg_{\widetilde{X}}(v_{i})-2\big)
=\sum_{i=1}^{n}\big(\deg_{X}(v_{i})-2\big) \leq n\cdot
\big(M\deg_{Z}(b)-2\big)\leq N\cdot \deg_{Y}(a)\cdot
\big(M\deg_{Z}(b)-2\big).$$

On the other hand, if $n\geq N\cdot \deg_{Y}(a)$, then
$$\sum_{i=1}^{n}
\big(\deg_{\widetilde{X}}(v_{i})-2\big)
=\sum_{i=1}^{n}\deg_{X}(v_{i})-2n \leq N\cdot \deg_{Y}(a)\cdot
\deg_{Z}(b)-2N\cdot \deg_{Y}(a) \leq N\cdot \deg_{Y}(a)\cdot
\big(\deg_{Z}(b)-2\big)\,.$$

Thus, in each case we have
\begin{equation}\label{3a}\sum_{i=1}^{n}
\big(\deg_{\widetilde{X}}(v_{i})-2\big)\leq N\cdot \deg_{Y}(a)\cdot
\big(M\deg_{Z}(b)-2\big)\,.\end{equation}
Since $\deg_{Y}(a)\geq
3$, or equivalently, $\deg_{Y}(a)\leq 3\big(\deg_{Y}(a)-2\big)$, it
follows that
\begin{align}
\sum_{i=1}^{n}
\big(\deg_{\widetilde{X}}(v_{i})-2\big)&\leq 3N\cdot
\big(\deg_{Y}(a)-2\big)\cdot
\big(M\deg_{Z}(b)-2\big)\nonumber\\
&\leq 3N\cdot
\big(\deg_{Y}(a)-2\big)\cdot\Big(M\big(\deg_{Z}(b)-2\big)+2(M-1)\Big)\nonumber \\
&= 3NM\cdot \big(\deg_{Y}(a)-2\big)\cdot
\big(\deg_{Z}(b)-2\big)+6N(M-1)\cdot \big(\deg_{Y}(a)-2\big)\nonumber\\
&\leq\big(3NM+6N(M-1)\big)\cdot
\big(\deg_{\widetilde{Y}}(a)-2\big)\cdot
\big(\deg_{\widetilde{Z}}(b)-2\big)\label{3b}.
\end{align}
This completes the proof of part 1) of the theorem.

2) To prove the second part, again by Lemma \ref{lem1}, it suffices
to show that
\begin{align}\label{ineq5} \sum_{v\in
\pi^{-1}(a,b)}\big(\deg_{\widetilde{X}}(v)-2\big)&\leq 3N
\big(\deg_{\widetilde{Y}}(a)-2\big) \cdot
\big(\deg_{\widetilde{Z}}(b)-2\big),
\end{align}
for each pair of vertices $(a,b)\in Y\times Z$. Proceeding exactly
as before, we distinguish three cases. In Case 1, where both $a$ and
$b$ are non-degenerate, we get the same inequality. In Cases 2 and
3, by Lemma \ref{lem2} (2), we can now use $1$ instead of $M$. Thus
in Cases 2 and 3, we obtain respectively (from \ref{2a}-\ref{2b} and
\ref{3b}) the inequalities
\begin{align}\sum_{i=1}^{n}\big(\deg_{\widetilde{X}}(v_{i})-2\big)&\leq N
\big(\deg_{\widetilde{Y}}(a)-2\big) \cdot
\big(\deg_{\widetilde{Z}}(b)-2\big)
\end{align}
and
\begin{align}\sum_{i=1}^{n}\big(\deg_{\widetilde{X}}(v_{i})-2\big)&\leq 3N
\big(\deg_{\widetilde{Y}}(a)-2\big) \cdot
\big(\deg_{\widetilde{Z}}(b)-2\big).
\end{align}
It remains only to consider the case when both $a$ and $b$ are
degenerate (in which case we are in Case 3) and $\deg_{Y}(a)=2$,
where $a$ is the vertex of maximal degree. If $\deg_{Y}(a)=2$, then
$\deg_{Z}(b)=2$ too, and inequality \ref{ineq5} follows since, by
Lemma \ref{lem2} (2),
$\deg_{\widetilde{X}}(v_{i})=\deg_{X}(v_{i})\leq
\min\{\deg_{Y}(a),\deg_{Z}(b) \}$ for each $i$.
\end{proof}

\begin{cor}(\cite[Theorem 1]{Za1})
Let $G$ be a group acting on a tree $T$ with finite quotient and
finite stabilizers of edges and let $H$, $K$ be finitely generated
subgroups of $G$ which intersect trivially each vertex stabilizer
(and hence they are free groups). Then $H\cap K$ is finitely
generated and
$$\overline{r}(H\cap K)\leq 6N\cdot\overline{r}(H)\cdot
\overline{r}(K),$$ where $N=\max\big\{|G_{x}\cap HK |\,:\,x\in
ET\big\}$.
\end{cor}

\begin{cor} Let $G$ be a group acting on a tree $T$ with finite
quotient, finite stabilizers of edges and infinite vertex
stabilizers. If $H$ and $K$ are subgroups of finite index in $G$,
then $$\overline{C}_{T}(H\cap K)\leq 2N\cdot\overline{C}_{T}(H)\cdot
\overline{C}_{T}(K).$$
\end{cor}
\begin{proof}
If the $G$-stabilizer of every vertex is infinite and both $H$ and
$K$ are of finite index in $G$, then each vertex stabilizer is also
infinite under the action of $H$ or $K$ (being of finite index in
the corresponding $G$-stabilizer) and thus Cases 2 and 3 do not
occur.
\end{proof}

Following \cite{DI}, given a group $G$, we define
$a_{3}(G)=\min\big\{|\Gamma|\,:\,\Gamma \textrm{ is a subgroup of }
G \textrm{ with } |\Gamma|\geq 3 \big\}$ and
$\theta(G)=\Big\{\frac{a_{3}(G)}{a_{3}(G)-2}\Big\}\in [1,3]$, where
$\frac{\infty}{\infty-2}:=1$.

In the sequel, we prove that if $H,K$ act freely on the edges, then
the coefficient $6$ in the above theorem can be replaced by a number
$2\theta$, where $\theta\in [1,3]$, by imposing some extra
hypotheses on the structure of $G$.

Let $G_{i}$, $i\in I$, be a family of groups together with a group
$A$, let $\phi_{i}:A\longrightarrow G_{i}$ be a family of
monomorphisms and let $\ast_{A}G_{i}$ be the amalgamated free
product of $G_{i}$'s with amalgamated subgroup $A$ (with respect to
$\phi_{i}$). We can think of each $\phi_{i}$ as an inclusion. Let
$F$ be a free group and let $G=\ast_{A}G_{i}\ast F$ be the free
product of $F$ and $\ast_{A}G_{i}$. We construct a graph of groups
$(\mathcal{G},\Psi)$ with fundamental group $G$ as follows. The
graph $\Psi$ consists of a wedge of open edges
$e_{i}=[u_{0},u_{i}],i\in I$ (i.e. one for each factor $G_{i}$ and
distinct endpoints $u_{0}$ and $u_{i},i\in I$, $0\notin I$),
together with a wedge of loops $l_{j}$, one for each free generator
of $F$, attached at a vertex $u_{0}$ with vertex group $A$. To each
edge $e_{i}$ we associate the group $A$, to each loop $l_{j}$ we
associate the trivial group and to each vertex $u_{i}$ we associate
the group $G_{i}$. We denote by $T$ the corresponding universal
tree.

\firstThmTwo
\begin{proof}
We proceed as in the proof of Theorem \ref{ThmOne}. With the
notation of that proof, we have to prove that
\begin{equation}\label{ineq4}\sum_{i=1}^{n} \big(\deg_{\widetilde{X}}(v_{i})-2\big) \leq
\theta \cdot N \cdot \big(\deg_{\widetilde{Y}}(a) -2\big)\cdot
\big(\deg_{\widetilde{Z}}(b) -2\big) \end{equation}
 for any pair of
vertices $(a,b)\in Y\times Z$ (recall that $\{v_{1},\ldots, v_{n}\}$
denotes the vertices of $\pi^{-1}(a,b)$). Since $H$, $K$ act freely
on the edges of $T$, it follows, by Lemma \ref{lem2} (2), that we
can use $1$ instead of $M$. Suppose first that at least one of $a$
and $b$ is non-degenerate. Then the arguments of Cases 1, 2 of the
proof of Theorem \ref{ThmOne} apply to show that
\begin{equation}\label{ineq6}\sum_{i=1}^{n} \big(\deg_{\widetilde{X}}(v_{i})-2\big) \leq
N \cdot \big(\deg_{\widetilde{Y}}(a) -2\big)\cdot
\big(\deg_{\widetilde{Z}}(b) -2\big) \end{equation} and inequality
\ref{ineq4} holds. Thus it suffices to consider the case where $a$
is $H$-degenerate and $b$ is $K$-degenerate (i.e. Case 3 in the
proof of Theorem \ref{ThmOne}). In this case we have $\deg_{Y}(a)
=\deg_{\widetilde{Y}}(a)$, $\deg_{Z}(b) = \deg_{\widetilde{Z}}(b)$
and  $\deg_{X}(v_{i}) = \deg_{\widetilde{X}}(v_{i})$, while by Lemma
\ref{lem2} (2), $\deg_{X}(v_{i})\leq \min\{\deg_{Y}(a),\deg_{Z}(b)
\}$. For each $i\in\{1,\ldots, n\}$, choose a vertex, $w_{i}$, of
$T$, so that $[w_{i}]_{H\cap K}=v_{i}\in \pi^{-1}(a,b)$. Note that
all $w_{1},\ldots,w_{n}$ lie in the same $G$-orbit. There are two
subcases to consider.

(i) $w_{i}$ and $u_{0}$ are in the same $G$-orbit, i.e.
$w_{i}=g_{i}u_{0}$ for some $g_{i}\in G$. Notice that
$H_{w_{i}}=K_{w_{i}}=g_{i}Ag_{i}^{-1}\cap H=g_{i}Ag_{i}^{-1}\cap
K=1$. If one of the vertices $a$ or $b$ has valence $2$, then (each)
$v_{i}$ has valence $2$ as well and inequality \ref{ineq4} is
obvious. If both $a$ and $b$ have valence at least $3$, then
\begin{align}\sum_{i=1}^{n} \big(\deg_{\widetilde{X}}(v_{i})-2\big)& \leq
n\cdot \big(\deg_{\widetilde{Y}}(a) -2\big)\leq |H_{w_{i}}\backslash
G_{w_{i}}\cap HK / K_{w_{i}}| \big(\deg_{\widetilde{Y}}(a)
-2\big)\cdot (3 -2)\nonumber \\
 &\leq N\cdot
\big(\deg_{\widetilde{Y}}(a) -2\big)\cdot
\big(\deg_{\widetilde{Z}}(b) -2\big)\leq \theta(G)\cdot N\cdot
\big(\deg_{\widetilde{Y}}(a) -2\big)\cdot
\big(\deg_{\widetilde{Z}}(b) -2\big).\nonumber\end{align}

(ii) $w_{i}$ and $u_{j}$ are in the same $G$-orbit for some $j\in
I$, i.e. there exists $g_{i}\in G$ such that $w_{i}=g_{i}u_{j}$. As
before, we may assume that both $a$ and $b$ have valence at least
$3$.

If $n\leq N$, then
\begin{align}
\sum_{i=1}^{n} \big(\deg_{\widetilde{X}}(v_{i})-2\big)\leq N\cdot
\big(\deg_{\widetilde{Y}}(a) -2\big)\cdot (3-2)\leq N\cdot
\big(\deg_{\widetilde{Y}}(a) -2\big)\cdot
\big(\deg_{\widetilde{Z}}(b) -2\big). \nonumber\end{align}

Suppose now that $n>N$. Let $\mathcal{R}=\{g_{\lambda}\}_{\lambda\in
\Lambda}$ be a set of representatives of left cosets of $A$ in
$G_{j}$. From the construction of $X=T/G$, the stars of different
vertices of $\pi^{-1}(a,b)$ are disjoint, while each edge in the
star of $[g_{i}u_{j}]_{\Gamma}$, where $\Gamma=H\cap K,H \textrm{ or
} K$, is of the form $[xe_{j}]_{\Gamma}$ and its terminal vertex is
$[g_{i}u_{j}]_{\Gamma}$. It follows that there is $\gamma\in \Gamma$
such that $\gamma x u_{j}=g_{i}u_{j}$ and thus $g_{i}^{-1}\gamma
x\in G_{j}$. If we write $g_{i}^{-1}\gamma x$ as $g_{\lambda(x)}a$,
where $g_{\lambda(x)}\in \mathcal{R}$ and $a\in A$, then
$[xe_{j}]_{\Gamma}=[\gamma^{-1}g_{i}g_{\lambda(x)}a
e_{j}]_{\Gamma}=[g_{i}g_{\lambda(x)}a
e_{j}]_{\Gamma}=[g_{i}g_{\lambda(x)} e_{j}]_{\Gamma}$. It follows
that there exists a subset $\mathcal{R}^{i}_{\Gamma}$ of
$\mathcal{R}$ such that
$Star([w_{i}]_{\Gamma})=\big\{[g_{i}g_{\lambda}
e_{j}]_{\Gamma}\,:\,g_{\lambda}\in\mathcal{R}^{i}_{\Gamma}\big\}$
and $|\mathcal{R}^{i}_{\Gamma}|=|Star([w_{i}]_{\Gamma})|$. In
particular, $|\mathcal{R}^{i}_{H}|=|Star_{Y}(a)|$ and
$|\mathcal{R}^{i}_{K}|=|Star_{Z}(b)|$.

Fix $i\in \{1,\ldots,n\}$. For any $k\in\{1,\ldots,n\}$, let $C_{k}$
be the subset of $\mathcal{R}_{H}\times
\mathcal{R}_{K}:=\mathcal{R}^{i}_{H}\times \mathcal{R}^{i}_{K}$
consisting of all pairs $(g_{\lambda},g_{\mu})$ such that
$\big([g_{i}g_{\lambda} e_{j}]_{H},[g_{i}g_{\mu} e_{j}]_{K}\big)$ is
the image under $\pi$ of some edge $[xe_{j}]_{H\cap K}$ in the star
of $[g_{k}u_{j}]_{H\cap K}=v_{k}$ in $X$. Let $\phi:G_{j}\rightarrow
G_{j}/A$ denote the natural epimorphism. Note that the restriction
of $\phi$ on $\mathcal{R}$ is a bijection.

We will show that
$\phi(C_{k})=\big\{\big(\phi(g_{\lambda}),\phi(g_{\mu})\big):\,
(g_{\lambda},g_{\mu})\in C_{k}\big\}$ is a \textsl{single-quotient}
subset of $\phi(\mathcal{R}_{H})\times \phi(\mathcal{R}_{K})$, in
the terminology of \cite{DI}, i.e. that the product
$\phi(g_{\lambda})\cdot \phi(g_{\mu})^{-1}$ is constant for all
pairs $(g_{\lambda},g_{\mu})\in C_{k}$. Suppose that $y_{1}$ and
$y_{2}$ are edges in the star of $v_{k}$ and that
$\pi(y_{t})=\big([g_{i}g_{\lambda(t)} e_{j}]_{H},[g_{i}g_{\mu(t)}
e_{j}]_{K}\big)$, $t=1,2$. We want to show that
$\phi(g_{\lambda(1)})\cdot
\phi(g_{\mu(1)})^{-1}=\phi(g_{\lambda(2)})\cdot
\phi(g_{\mu(2)})^{-1}$. From the above analysis, we can write
$y_{t}=[g_{k}g_{s(t)} e_{j}]_{H\cap K}$ for some
$g_{s(t)}\in\mathcal{R}^{k}_{H\cap K}$, $t=1,2$, and thus
$\pi(y_{t})=\big([g_{k}g_{s(t)} e_{j}]_{H},[g_{k}g_{s(t)} e_{j}]_{
K}\big)$. It follows that $\big([g_{k}g_{s(1)}
e_{j}]_{H},[g_{k}g_{s(1)} e_{j}]_{ K}\big)=\big([g_{i}g_{\lambda(1)}
e_{j}]_{H},[g_{i}g_{\mu(1)} e_{j}]_{K}\big)$, and that
$\big([g_{k}g_{s(2)} e_{j}]_{H},[g_{k}g_{s(2)} e_{j}]_{
K}\big)=\big([g_{i}g_{\lambda(2)} e_{j}]_{H},[g_{i}g_{\mu(2)}
e_{j}]_{K}\big)$. Hence there are $h_{1},h_{2}\in H$,
$k_{1},k_{2}\in K$ and $a_{1},a_{2},a_{3},a_{4}\in A$ such that
$$\begin{array}{ll}
 g_{k}g_{s(1)}=h_{1}g_{i}g_{\lambda(1)}a_{1}, & g_{k}g_{s(2)}=h_{2} g_{i}g_{\lambda(2)}a_{3} \\
 g_{k}g_{s(1)}=k_{1}g_{i}g_{\mu(1)}a_{2}, & g_{k}g_{s(2)}=k_{2}g_{i}g_{\mu(2)}a_{4}   \\
\end{array}$$
By normality of $A$ in $G_{j}$, the stabilizer of any edge in the
star of $w_{i}$ is equal to $g_{i}Ag_{i}^{-1}$. Therefore our
assumption that $w_{i}$ is $H,K$ degenerate implies that $H\cap
G_{w_{i}}=H\cap g_{i}Ag_{i}^{-1}$ and $K\cap G_{w_{i}}=K\cap
g_{i}Ag_{i}^{-1}$. Now, from the first two equalities above we
deduce that
\begin{equation}\label{eq5} h_{2}^{-1}h_{1}=
g_{i}g_{\lambda(2)}a_{3}g_{s(2)}^{-1}g_{s(1)}a_{1}^{-1}g_{\lambda(1)}^{-1}g_{i}^{-1}\in
H\cap g_{i}G_{j}g_{i}^{-1}=H\cap G_{w_{i}}=H\cap
g_{i}Ag_{i}^{-1}=1,\end{equation} while from the last two
\begin{equation}\label{eq6} k_{2}^{-1}k_{1}=g_{i}g_{\mu(2)}a_{4}g_{s(2)}^{-1}g_{s(1)}a_{2}^{-1}g_{\mu(1)}^{-1}g_{i}^{-1}\in
K\cap g_{i}G_{j}g_{i}^{-1}=K\cap G_{w_{i}}=K\cap
g_{i}Ag_{i}^{-1}=1.\end{equation} The above relations imply that
$g_{\lambda(2)}a_{3}g_{s(2)}^{-1}g_{s(1)}a_{1}^{-1}g_{\lambda(1)}^{-1}=1$
and
$g_{\mu(2)}a_{4}g_{s(2)}^{-1}g_{s(1)}a_{2}^{-1}g_{\mu(1)}^{-1}=1$.
Thus,
$g_{\lambda(1)}g_{\mu(1)}^{-1}=g_{\lambda(2)}a_{3}g_{s(2)}^{-1}g_{s(1)}a_{1}^{-1}
a_{2}g_{s(1)}^{-1}g_{s(2)}a_{4}^{-1}g_{\mu(2)}^{-1}$, from which it
follows that $\phi(g_{\lambda(1)})\cdot
\phi(g_{\mu(1)})^{-1}=\phi(g_{\lambda(2)})\cdot
\phi(g_{\mu(2)})^{-1}$.

Our aim is to apply \cite[Corollary 3.5]{DI}, which requires
pairwise-disjoint, single-quotient subsets. Note that if the
intersection $C_{k}\cap C_{s}$ is nonempty, then there are edges
$y_{1}$ and $y_{2}$ in $Star_{X}(v_{k})$ and $Star_{X}(v_{s})$,
respectively, such that $\pi(y_{1})=\pi(y_{2})$. Thus, for each
$k=1,\ldots,n$, we choose a subset $F_{k}$ of $Star_{X}(v_{k})$ with
$|F_{1}|+\cdots+|F_{n}|$ maximum such that the restriction of $\pi$
on the union $\cup_{k=1}^{n}F_{k}$ is an injection. In particular,
they are pairwise-disjoint. Since the inverse image of any edge of
$Y\times Z$ under $\pi$ contains at most $N$ elements, we have
$|Star_{X}(v_{1})|+\cdots +|Star_{X}(v_{n})|\leq N
\big(|F_{1}|+\cdots+|F_{n}|\big)$. If $C_{F(k)}$ denotes the subset
of $C_{k}$ corresponding to edges of $F_{k}$, then
$C_{F(1)},\ldots,C_{F(n)}$ are pairwise-disjoint. It follows that
$\phi(C_{F(1)}),\ldots,\phi(C_{F(n)})$ are pairwise-disjoint,
single-quotient subsets of $\phi(\mathcal{R}_{H})\times
\phi(\mathcal{R}_{K})$ and \cite[Corollary 3.5]{DI} applies to show
that
$$\sum_{k=1}^{n}\big(|\phi(C_{F(k)})|-2\big)\leq
\theta(G_{j}/A)\cdot \big(|\phi(\mathcal{R}_{H})|-2 \big)\cdot
\big(|\phi(\mathcal{R}_{K})|-2 \big).$$ Finally
\begin{align}\sum_{k=1}^{n}
\big(\deg_{\widetilde{X}}(v_{k})-2\big)&= \sum_{k=1}^{n}
\big(\deg_{X}(v_{k})-2\big)= \sum_{k=1}^{n} |Star_{X}(v_{k})|-2n
\leq N\cdot \sum_{k=1}^{n} |F_{k}|-2N \nonumber \\
&= N\cdot \sum_{k=1}^{n} \big(|C_{F(k)}|-2\big)=N\cdot
\sum_{k=1}^{n} \big(|\phi(C_{F(k)})|-2\big)\nonumber
\\
& \leq N\cdot \theta(G_{j}/A)\cdot \big(|\phi(\mathcal{R}_{H})|-2
\big)\cdot \big(|\phi(\mathcal{R}_{K})|-2 \big)\nonumber \\
& =N\cdot \theta(G_{j}/A)\cdot \big(|\mathcal{R}_{H}|-2 \big)\cdot
\big(|\mathcal{R}_{K}|-2 \big) \nonumber \\
 &\leq  N \cdot \theta \cdot
\big(\deg_{\widetilde{Y}}(a) -2\big)\cdot
\big(\deg_{\widetilde{Z}}(b) -2\big).\nonumber\end{align} This
completes the proof.
\end{proof}
\begin{rem}\label{rem2}
The analogous theorem with the same proof is valid for fundamental
groups of graphs of groups $(\mathcal{G},\Psi)$ defined as follows.
The subject graph $\Psi$ is the same as the one defined previously
(prior to Theorem \ref{thm2}). To the terminal vertex $u_{i}$ of
$e_{i}$ we associate the group $G_{i}$, to the common initial vertex
of $e_{i}$'s we associate the finite group $A$, and to each open
edge $e_{i}$ we associate a subgroup $A_{i}$ of $A$ normally
embedded in $G_{i}$ such that $A_{i_{0}}=A$ for some $i_{0}$ (this
means that the ``central" vertex is $G$-degenerate and thus
$G_{w_{i}}=gAg^{-1}$ in Case (i) of the proof). To each loop we
associate the trivial group. We need normality of $A_{i}$ in $G_{i}$
in order to make the natural map $G_{i}\rightarrow G_{i}/A_{i}$ a
homomorphism (and thus the same arguments in Case (ii) work equally
well to this more general setting).
\end{rem}

As a corollary, we obtain the main result of Ivanov in \cite{Iv} (in
fact our proof can be slightly modified to generalize \cite[Theorem
6.3]{DI} as well).

\begin{cor}
Suppose that $H_{1}$, $H_{2}$ are subgroups of a free product
$G=\ast_{a\in I}G_{a}$ and $H_{1}$, $H_{2}$ have finite Kurosh rank
$K(H_{1})$, $K(H_{2})$. Then the intersection $H_{1}\cap H_{2}$ also
has finite Kurosh rank and
$$\overline{K}r(H_{1}\cap H_{2})\leq 2\cdot \theta(G) \cdot \overline{K}r(H_{1})\cdot \overline{K}r(H_{2}). $$
In particular, if $G$ is torsion-free (or more generally, every
finite subgroup of $G$ has order at most $2$), then
$$\overline{K}r(H_{1}\cap H_{2}) \leq 2\cdot \overline{K}r(H_{1})\cdot
\overline{K}r(H_{2}). $$
\end{cor}
\begin{proof}
By Lemma \ref{lem0} and the comments preceding it, the subgroup
Kurosh rank is equal to the Kurosh rank with respect to $T$ (i.e.
$Kr(\cdot)=K_{T}(\cdot)$, where $T$ is as above) and finite Kurosh
rank implies tameness.
\end{proof}

In the case of free products with a finite, normal subgroup
amalgamated, we can use the same arguments to improve the bound for
the complexity of the intersection of tame subgroups.

Let $G_{i}$, $i\in I$, be a family of groups together with a group
$A$ and let $G=\ast_{A}G_{i}$ be the amalgamated free product of
$G_{i}$'s with amalgamated subgroup $A$ (with respect to a family of
monomorphisms, regarded as inclusions). We construct a tree of
groups $(\mathcal{G},T_{0})$ with fundamental group $G$ as usual.
The tree $T_{0}$ consists of a wedge of open edges
$e_{i}=[u_{0},u_{i}],i\in I$ (one for each factor $G_{i}$) attached
at a vertex $v_{0}$ (where $0\notin I$) with vertex group $A$. To
each edge we associate the group $A$ and to each vertex $v_{i}$ we
associate the group $G_{i}$. We denote by $T$ the corresponding
universal tree.

\begin{thm}\label{thm3} Let $G=\ast_{A}G_{i}$ be the amalgamated free product of $G_{i}$'s
with a finite and normal amalgamated subgroup $A$. We consider the
action of $G$ on $T$ defined above. If $H$ and $K$ are tame
subgroups (with respect to $T$) of $G$, then $H\cap K$ is tame and
$$\overline{C}_{T}(H\cap K)\leq 2\cdot \theta \cdot |A\cap HK|\cdot\overline{C}_{T}(H)\cdot
\overline{C}_{T}(K),$$ where $\theta=\max\{\theta(G_{i}/A)\,:\,i\in
I\}$.
\end{thm}
\begin{proof} The proof is exactly the same as the proof of Theorem \ref{thm2}.
There are two things to note:
\begin{itemize}
\item[(a)] The normality of $A$ in $G$ and the fact that $v_{0}$
is a $G$-degenerate vertex imply that for each subgroup $B$ of $G$
the $B$-stabilizer of the star of any $B$-degenerate vertex $v$ of
$T$ is equal to $B_{v}$ and therefore Lemma \ref{lem2} (2) applies
(i.e we can again use $1$ instead of $M$ to obtain inequality
\ref{ineq6}).
\item[(b)] Using the notation of the proof of Theorem \ref{thm2},
the relations \ref{eq5} and \ref{eq6} now give
$g_{\lambda(2)}a_{3}g_{s(2)}^{-1}g_{s(1)}a_{1}^{-1}g_{\lambda(1)}^{-1}\in
A$ and
$g_{\mu(2)}a_{4}g_{s(2)}^{-1}g_{s(1)}a_{2}^{-1}g_{\mu(1)}^{-1}\in
A$. Since $A$ is the kernel of $\phi$, we again conclude that
$\phi(g_{\lambda(1)})\cdot
\phi(g_{\mu(1)})^{-1}=\phi(g_{\lambda(2)})\cdot
\phi(g_{\mu(2)})^{-1}$.
 \end{itemize}\end{proof}

\begin{rem}
In general, there are examples (see \cite{Iv,Za1}) showing that the
bounds obtained in the previous two theorems are sharp.
\end{rem}

\noindent
Department of Mathematics\\
National and Kapodistrian University of Athens\\
Panepistimioupolis, GR-157 84, Athens, Greece\\
{\it e-mail}: klentzos@math.uoa.gr\\
{\it e-mail}: msykiot@math.uoa.gr


\begin{thebibliography}{ABC}
\bibitem{AnMS} Y. Antol\'{i}n, A. Martino and I. Schwabrow, Kurosh rank of
intersections of subgroups of free products of right-orderable
groups, Math. Res. Lett. \textbf{21}, No. 4 (2014), 649--661.
\bibitem{ArSS} V. Ara\'{u}jo, P. V. Silva and M. Sykiotis.
Finiteness results for subgroups of finite extensions, J. Algebra
\textbf{423} (2015), 592--614.
\bibitem{Ba} H. Bass, Covering theory for graphs of groups,
J. Pure Appl. Algebra \textbf{89} (1993), no. 1-2, 3--47.
\bibitem{BCK} R. G. Burns, T. C. Chau and S.-M. Kam, On the rank of intersection of subgroups of a
free product of groups, J. Pure App. Algebra \textbf{124} (1998),
31--45.
\bibitem{DicksDun} W. Dicks and M.J. Dunwoody, ``Groups acting on
graphs", Cambridge University Press, 1989.
\bibitem{DI} W. Dicks and S. V. Ivanov, On the intersection of free
subgroups in free products of groups, Math. Proc. Cambridge Phil.
Soc. \textbf{144} (2008), 511--534.
\bibitem{DI1} W. Dicks and S. V. Ivanov, On the intersection of free
subgroups in free products of groups with no $2$-torsion, Illinois
J. Math. \textbf{54} (2010), 223--248.
\bibitem{Fr} J. Friedman, Sheaves on graphs, their homological invar
iants, and a proof of the Hanna Neumann conjecture: with an appendix
by Warren Dicks, Memoirs Amer. Math. Soc.  \textbf{233}, no. 1100
(2014).
\bibitem{Ho} A. G. Howson, On the intersection of finitely generated free groups, J. London Math.
Soc. \textbf{29} (1954), 428--434.
\bibitem{Iv0} S. V. Ivanov, Intersecting free subgroups in free products of groups. Internat. J.
Algebra Comput. \textbf{11} (2001), no. 3, 281--290.
\bibitem{Iv} S. V. Ivanov, On the Kurosh rank of the intersection of subgroups in free products of
groups, Adv. Math. \textbf{218} (2008), no. 2, 465--484.
\bibitem{Mi}  I. Mineyev, Submultiplicativity and the Hanna Neumann conjecture. Ann. of Math. (2)
\textbf{175} (2012), 393--414.
\bibitem{Se} J.-P. Serre, ``Trees",  Springer-Verlag, Berlin-Heidelberg-New
York, 1980.
\bibitem{So} T. Soma, Intersection of finitely generated surface groups. J.
Pure Appl. Algebra \textbf{66} (1990), no. 1-3, 81--95.
\bibitem{Sy} M. Sykiotis, On subgroups of finite complexity in groups acting on
trees, J. Pure Appl. Algebra \textbf{200} (2005), 1--23.
\bibitem{Za1} A. Zakharov, On the rank of the intersection of free subgroups in virtually free
groups, J. Algebra \textbf{418} (2014), 29--43.
\bibitem{Za2} A. Zakharov, Intersecting free subgroups in free amalgamated
products of two groups with normal finite amalgamated subgroup, Mat.
Sb. \textbf{204} (2) (2013), 223--236.
\end{thebibliography}
\end{document}